\documentclass{amsart}

\usepackage{amsmath}
\usepackage{amssymb}
\usepackage{amsthm}
\usepackage[msc-links,abbrev]{amsrefs}
\usepackage{hyperref}
\usepackage[noabbrev,capitalize]{cleveref}
\usepackage{mathrsfs}
\usepackage{mathtools}
\usepackage{slashed}
\usepackage{enumitem}
\usepackage{stmaryrd}

\setenumerate{label=(\roman*)}

\DeclareMathOperator{\Id}{Id}

\DeclareMathOperator{\tr}{tr}
\DeclareMathOperator{\Vol}{Vol}
\DeclareMathOperator{\dvol}{dvol}

\DeclareMathOperator{\Ric}{Ric}

\newcommand{\defn}[1]{{\boldmath\bfseries#1}}

\newcommand{\grd}{g_{\mathrm{rd}}}
\newcommand{\ghyp}{g_{\mathrm{hyp}}}
\newcommand{\gfl}{g_{\mathrm{flat}}}

\newcommand{\Diff}{\mathrm{Diff}}

\newcommand{\cg}{\widetilde{g}}

\newcommand{\cu}{\widetilde{u}}

\newcommand{\cX}{\widetilde{X}}

\newcommand{\cgamma}{\widetilde{\gamma}}

\newcommand{\csigma}{\widetilde{\sigma}}

\newcommand{\cmG}{\widetilde{\mathcal{G}}}

\newcommand{\hg}{\widehat{g}}

\newcommand{\lp}{\langle}
\newcommand{\rp}{\rangle}
\newcommand{\lv}{\lvert}
\newcommand{\rv}{\rvert}

\newcommand{\mD}{\mathcal{D}}

\newcommand{\bN}{\mathbb{N}}

\newcommand{\bR}{\mathbb{R}}

\newcommand{\bZ}{\mathbb{Z}}

\def\sideremark#1{\ifvmode\leavevmode\fi\vadjust{\vbox to0pt{\vss
 \hbox to 0pt{\hskip\hsize\hskip1em
 \vbox{\hsize3cm\tiny\raggedright\pretolerance10000
 \noindent #1\hfill}\hss}\vbox to8pt{\vfil}\vss}}}

\newcommand{\suchthat}{\mathrel{}:\mathrel{}}

\newtheorem{theorem}{Theorem}[section]
\newtheorem{proposition}[theorem]{Proposition}
\newtheorem{lemma}[theorem]{Lemma}
\newtheorem{corollary}[theorem]{Corollary}

\theoremstyle{definition}

\theoremstyle{remark}
\newtheorem{remark}[theorem]{Remark}

\numberwithin{equation}{section}

\makeatletter
\@namedef{subjclassname@2020}{%
  \textup{2020} Mathematics Subject Classification}
\makeatother

\begin{document}

\title[The GJMS operators of a special Einstein product]{A factorization of the GJMS operators of special Einstein products and applications}
\author{Jeffrey S. Case}
\address{Department of Mathematics \\ Penn State University \\ University Park, PA 16802}
\email{jscase@psu.edu}
\author{Andrea Malchiodi}
\address{Scuola Normale Superiore \\ Piazza dei Cavalieri 7 \\ 56126 Pisa}
\email{andrea.malchiodi@sns.it}
\begin{abstract}
 We show that the GJMS operators of a special Einstein product factor as a composition of second- and fourth-order differential operators.
 In particular, our formula applies to the Riemannian product $H^{\ell} \times S^{d-\ell}$.
 We also show that there is an integer $D = D(k,\ell)$ such that if $d \geq D$, then for any special Einstein product $N^\ell \times M^{d-\ell}$, the Green's function for the GJMS operator of order $2k$ is positive.
 As a result, these products give new examples of closed Riemannian manifolds for which the $Q_{2k}$-Yamabe problem is solvable.
\end{abstract}
\maketitle

\section{Introduction}
\label{sec:intro}

Let $(X^d,g)$ be a Riemannian manifold and let $k \in \bN$ be a positive integer;
if~$d$ is even, assume additionally that $k \leq d/2$.
The \defn{GJMS operator}~\cite{GJMS1992} of order~$2k$ is a formally self-adjoint conformally covariant differential operator $P_{2k}^g$ with leading-order term $(-\Delta_g)^k$.
Here \defn{conformal covariance} means that
\begin{equation}
 \label{eqn:conformally-covariant}
 P_{2k}^{e^{2\Upsilon}g}(u) = e^{-\frac{d+2k}{2}\Upsilon} P_{2k}^g \left( e^{\frac{d-2k}{2}\Upsilon}u \right)
\end{equation}
for all $u,\Upsilon \in C^\infty(X)$.
Explicit formulas for $P_{2k}^g$ are known in some low-order cases~\cites{LeeParker1987,Paneitz1983,Branson1995,Juhl2013,Wunsch1986,GoverPeterson2003}, but only a recursive formula is available for the general-order case~\cites{Juhl2013,FeffermanGraham2013}.
However, there is a simple formula on Einstein manifolds ~\cites{Gover2006q,FeffermanGraham2012} for these operators:
if $\Ric_g = (d-1)\lambda g$, then
\begin{equation}
 \label{eqn:einstein-factorization}
 P_{2k}^g = \prod_{j=1}^{k} \left( -\Delta_g + \frac{(d+2j-2)(d-2j)}{4}\lambda \right) .
\end{equation}
In particular, $P_{2k}=(-\Delta)^k$ on Euclidean space.
Our convention is that $-\Delta \geq 0$.

The factorization~\eqref{eqn:einstein-factorization} plays an important role in Beckner's proof~\cite{Beckner1993} of the sharp Sobolev inequalities.
It also motivated studies of the Paneitz--Branson-type operators~\cite{DjadliHebeyLedoux2000}.
These observations lead us to look for other situations where the GJMS operators have a nice factorization.

A \defn{special Einstein product} is a Riemannian product $(M^\ell \times N^{d-\ell} , g \oplus h)$ of Einstein manifolds such that $\Ric_g = -(\ell-1)\lambda g$ and $\Ric_h = (d-\ell-1)\lambda h$ for some constant $\lambda > 0$.
The assumption $\lambda>0$ is for convenience;
the case $\lambda=0$ yields Ricci flat manifolds and the case $\lambda<0$ is equivalent to our definition after switching the order of the factors.

The simplest examples of special Einstein products are Riemannian products of an $\ell$-dimensional hyperbolic manifold and a $(d-\ell)$-dimensional spherical manifold;
these are precisely the locally conformally flat special Einstein products~\cite{Besse}.
When $\ell=1$, these are quotients of the Riemannian product $S^1 \times S^{d-1}$.
These products arise, for example, as models for the singular $Q$-curvature prescription problem~\cites{MazzeoPacard1996,AoChanDelatorreFontelosGonzalezWei2019,ChangHangYang2004,MazzeoSmale1991,AndradeWei2022,LuoWeiZou2021,LuoWeiZou2022} and as examples of conformal manifolds with geometrically distinct metrics of constant $Q$-curvature~\cites{Schoen1989,BettiolPiccioneSire2021,BettiolPiccione2018,BettiolGonzalezMaalaoui2023}.
More generally, Case and Chang~\cite{CaseChang2013} used the special Einstein product $M^\ell \times S^{d-\ell}$ of a Poincar\'e--Einstein manifold $(M^\ell,g_+)$ with the round $(d-\ell)$-sphere to realize the fractional GJMS operators~\cite{GrahamZworski2003} as generalized Dirichlet-to-Neumann operators.
To make this realization, they needed the factorization
\begin{equation}
 \label{eqn:case-chang-factorization}
 P_{2k}^{g_+ \oplus h} \circ \pi^\ast = \pi^\ast \circ \prod_{j=0}^{k-1} \left( -\Delta_{g_+} - \frac{(2\ell-d+2k-4j-2)(d-2k+4j)}{4} \right) ,
\end{equation}
$\pi \colon M^\ell \times S^{d-\ell} \to M^\ell$, of the restriction of GJMS operator $P_{2k}$ to lifts of smooth functions on $M^\ell$.
One can make sense of this formula for fractional values of $d$ using the language of smooth metric measure spaces~\cite{Khaitan2022}.

Ao et al.\ used~\cite{AoChanDelatorreFontelosGonzalezWei2019} the Helgason--Fourier transform~\cite{Helgason2008} to prove a spectral formula for the fractional GJMS operators of the locally conformally flat special Einstein products.
While their formula completely describes the spectrum, it is difficult to use for other purposes, such as determining when $P_{2k}$ satisfies the Strong Maximum Principle.

The first goal of this article is to prove that the GJMS operators of special Einstein products factor as a composition of second- and fourth-order differential operators.

\begin{theorem}
 \label{factorization}
 Let $(M^\ell \times N^{d-\ell}, g \oplus h)$ be a special Einstein product.
 Let $k \in \bN$;
 if $d$ is even, assume additionally that $k \leq \frac{d}{2}$.
 Denote
 \begin{multline}
  \label{eqn:defn-Dt}
  D_t := (\Delta_g + \Delta_h)^2 - \frac{(d-2)(d-2\ell)}{2}\lambda(\Delta_g+\Delta_h) - 2t^2\lambda(\Delta_g-\Delta_h) \\
   + \left( \left(\frac{d-2}{2}\right)^2 - t^2\right)\left( \left(\frac{d-2\ell}{2}\right)^2 - t^2\right)\lambda^2 .
 \end{multline}
 If $k$ is even, then
 \begin{subequations}
  \label{eqn:factorization}
  \begin{equation}
   \label{eqn:factorization-even}
   P_{2k}^{g \oplus h} = \prod_{s=0}^{(k-2)/2} D_{k-1-2s} .
  \end{equation}
  If $k$ is odd, then
  \begin{equation}
   \label{eqn:factorization-odd}
   P_{2k}^{g \oplus h} = \left( -\Delta_g - \Delta_h + \frac{(d-2)(d-2\ell)}{4}\lambda \right) \prod_{s=0}^{(k-3)/2} D_{k-1-2s} ,
  \end{equation}
 \end{subequations}
 with the convention that the empty product is the identity.
\end{theorem}

The cases $\ell=0$ and $\ell=d$ are allowed in \cref{factorization}, with the convention $\Delta=0$ on zero-dimensional manifolds.
In this case, Equation~\eqref{eqn:factorization} recovers the factorization~\eqref{eqn:einstein-factorization} of the GJMS operators at Einstein manifolds.

The restriction $k \leq \frac{d}{2}$ in \cref{factorization} is required because of the constraint in the definition of the GJMS operators.
This restriction can be removed for locally conformally flat special Einstein products;
see \cref{lcf-remove-restriction} below for details.

We give a direct proof of \cref{factorization} using the explicit ambient metric of a special Einstein product~\cite{GoverLeitner2009}.
The main point is that the ambient Laplacian is well-behaved on a particular family of extensions (cf.\ \cites{Khaitan2022,Matsumoto2013}), leading to a formula for $P_{2k}$ as the sum of compositions of second-order operators; see \cref{main-thm} below.
Using the decomposition of functions on $(N^{d-\ell},h)$ in terms of eigenvalues of $\Delta_h$, we show that this sum is equivalent to a spectral formula involving the Gamma function;
see \cref{mar-formula} below.
We then expand this spectral formula to recover the factorization~\eqref{eqn:factorization}.
Our spectral formula coincides with the formula of Ao et al.~\cite{AoChanDelatorreFontelosGonzalezWei2019} in the case of local GJMS operators on locally conformally flat special Einstein products.
While our formula is also valid on more general special Einstein products, it is unlikely that the fractional order formula of Ao et al.\ generalizes as well, due to the dependence of the fractional GJMS operators on the choice of Poincar\'e--Einstein fill-in.

It would be interesting to determine if one can also produce ``nice'' formulas for the GJMS operators of other manifolds for which an explicit ambient metric is known~\cites{LeistnerNurowski2010,LeistnerNurowski2012,AndersonLeistnerNurowski2020}.

One reason to find explicit formulas for the GJMS operators is to better understand their analytic properties.
For example, we say that the GJMS operator $P_{2k}$ satisfies the \defn{Strong Maximum Principle} if $P_{2k}u \geq 0$ implies $u>0$ or $u=0$.
In general, $P_{2k}$ does not have this property;
e.g.\ $P_2$ satisfies the Strong Maximum Principle if and only if the Yamabe constant is positive.
The factorization~\eqref{eqn:einstein-factorization} implies that if $d > 2k$ and $(X^d,g)$ is a closed conformally Einstein manifold with positive Yamabe constant, then $P_{2k}$ satisfies the Strong Maximum Principle.
More strikingly, Gursky and Malchiodi~\cite{GurskyMalchiodi2014} and Hang and Yang~\cite{HangYang2014}, used the local formula for the Paneitz operator to find sufficient conditions for $P_4$ to satisfy the Strong Maximum Principle.

The second goal of this article is to give a sufficient condition for the GJMS operators of a special Einstein product to satisfy the Strong Maximum Principle.

\begin{theorem}
 \label{positivity}
 Let $k,\ell$ be nonnegative integers.
 There is a constant $D = D(k,\ell)$ such that if $d \geq D$ and $(M^\ell \times N^{d-\ell}, g \oplus h)$ is a closed special Einstein product, then $P_{2k}$ satisfies the Strong Maximum Principle and its Green's function is positive.
\end{theorem}

While we do not know the minimal value of $D(k,\ell)$ in general, \cref{factorization} implies that if $d = 2(k+\ell-1)$ and $\ell \geq 1$, then constants are in the kernel of $P_{2k}$ on the special Einstein product $(M^\ell \times N^{d-\ell}, g \oplus h)$.
Hence $D(k,\ell) \geq 2k + 2\ell - 1$.
The difficulty in computing $D(k,\ell)$ is that in our proof it corresponds to the largest root of a family of fifth-order polynomials.
However, two special cases are known.
First, the factorization~\cites{FeffermanGraham2012,Gover2006q} of the GJMS operators at Einstein manifolds implies that $D(k,0)=2k+1$.
Second, known results for the conformal Laplacian~\cite{LeeParker1987} and the Paneitz operator~\cite{GurskyMalchiodi2014} imply that $D(k,\ell) = 2k+2\ell-1$ for $k=1$ and $k=2$, respectively.
Also, our proof of \cref{positivity} implies that $D(3,1) \leq 48$;
see \cref{P6-remark} for further discussion.

We prove \cref{positivity} by showing that each factor~\eqref{eqn:defn-Dt} of $P_{2k}$ satisfies the Strong Maximum Principle for $D$ sufficiently large.
This follows by an iterative procedure as in Gursky and Malchiodi's proof~\cite{GurskyMalchiodi2014} that the Paneitz operator satisfies the Strong Maximum Principle under certain assumptions, but without passing through lower-order geometric curvatures;
see~\cref{baby-maximum-principle} for a precise statement and further discussion.

Combining \cref{positivity} with results of Qing and Raske~\cite{QingRaske2006}, Mazumdar~\cite{Mazumdar2016}, and Mazumdar and V\'etois~\cite{MazumdarVetois2020} yields conformal representatives realizing the $Q_{2k}$-Yamabe constant
\begin{equation}
 \label{eqn:yamabe}
 Y_{2k}(X^d,[g]) := \inf_{\hg \in [g]} \left\{ \int_X Q_{2k}^{\hg} \dvol_{\hg} \suchthat \Vol_{\hg}(X) = 1 \right\} .
\end{equation}

\begin{theorem}
 \label{solvability}
 Let $k,\ell$ be nonnegative integers.
 There is a constant $D = D(k,\ell)$ such that for every special Einstein product $(M^\ell \times N^{d-\ell}, g \oplus h)$ with $d \geq D$, there is a metric $\sigma \in [g \oplus h]$ such that
 \begin{align*}
  \int_{M \times N} Q_{2k}^\sigma \dvol_\sigma & = Y_{2k}(M^\ell \times N^{d-\ell}, [ g \oplus h ] ) , \\
  \Vol_{\sigma}(M \times N) & = 1 .
 \end{align*}
\end{theorem}

In other words, the $Q_{2k}$-Yamabe Problem is solvable on special Einstein products provided $d$ is sufficiently large relative to $k$ and $\ell$.
We expect that the best constant from \cref{positivity} is equal to the best constant from \cref{solvability}, though our proof of \cref{solvability} requires also $D \geq 2k + 4$.
This is because we prove the estimate $Y_{2k}(M^\ell \times N^{d-\ell} , [ g \oplus h ] ) < Y_{2k}(S^d,[\grd])$ through the usual split into the locally conformally flat and non-locally conformally flat cases (cf.\ \cites{Aubin1976,Schoen1984,LeeParker1987,MazumdarVetois2020}).
In the case of locally conformally flat manifolds, results of Qing and Raske~\cite{QingRaske2006} and Mazumdar~\cite{Mazumdar2016} imply that the $Q_{2k}$-Yamabe problem is solvable on special Einstein manifolds with $d \geq 2k+2\ell-1$.
Since we lack a positive mass theorem, the remaining case requires $d \geq 2k+4$ so that the Weyl tensor is the leading term in the asymptotic expansion~\cite{MazumdarVetois2020} of the Green's function of the GJMS operator $P_{2k}$.

Andrade et al.~\cite{AndradeCasePiccioneWei2023} use \cref{solvability} and a covering argument to show that for each $k,m \in \bN$, there is a special Einstein product $(M^\ell \times N^{d-\ell}, g \oplus h)$, $d \geq D(k,\ell)$, which admits at least $m$ pairwise nonhomothetic metrics in the conformal class $[g \oplus h]$ of constant $Q_{2k}$-curvature.
A continuity argument of Bettiol, Gonz\'alez, and Maalaoui~\cite{BettiolGonzalezMaalaoui2023} implies that the same conclusion is true for the fractional $Q_{2\gamma}$-curvature for $\gamma$ sufficiently close to an integer.
Thus the nonuniqueness of solutions to the Yamabe problem is not a phenomenon of only the low-order $Q$-curvatures.

This article is organized as follows.

In \cref{sec:ambient} we recall the explicit ambient metric for a special Einstein product and relate it to the GJMS construction.

In \cref{sec:proof} we derive various formulas for the GJMS operators of a special Einstein product.
In particular, we prove \cref{factorization}.

In \cref{sec:maximum} we prove \cref{positivity,solvability}.

\section{The ambient space of a special Einstein product}
\label{sec:ambient}

Let $(X^d,g)$ be a Riemannian manifold.
There is~\cite{FeffermanGraham2012}*{Proposition~2.6 and Theorem~2.9} a unique \defn{straight and normal ambient metric}
\begin{equation*}
 \cg = 2\rho \, dt^2 + 2t \, dt \, d\rho + t^2g_\rho
\end{equation*}
on $\cmG := (0,\infty)_t \times M \times (-\varepsilon,\varepsilon)_\rho$ such that
\begin{enumerate}
 \item $\Ric_{\cg} = O(\rho^\infty)$, if $d$ is odd, and
 \item $\Ric_{\cg} = O^+(\rho^{(d-2)/2})$, if $d$ even,
\end{enumerate}
where $O^+(\rho^k)$ denotes the space of symmetric $(0,2)$-tensor fields $T$ on $\cmG$ such that $T \in O(\rho^k)$ and $g^{ij}T_{ij} \in O(\rho^{k+1})$, where $T_{ij}$, $g_{ij}$, and $g^{ij}$ denote the restrictions of $T$, $g_\rho$, and $g_\rho^{-1}$, respectively, to (co)vectors tangent to $\{ t \} \times M \times \{\rho\}$.
Uniqueness is modulo $O(\rho^\infty)$ if $d$ is odd, and modulo $O^+(\rho^{d/2})$ if $d$ is even.
Moreover, if $\hg := e^{2u}g$ is a conformal rescaling of $g$ and $\widetilde{\widehat{g}}$ is the straight and normal ambient metric for $\hg$, then~\cite{FeffermanGraham2012}*{Theorem~2.3} there is a diffeomorphism $\Phi \in \Diff(\cmG)$ such that $\Phi\rv_{\rho=0}=\Id$ and $\Phi^\ast \widetilde{\widehat{g}} \equiv \cg$ modulo the indeterminancy of a straight and normal ambient metric.

Let $(\cmG,\cg)$ be the straight and normal ambient space for $(X^d,g)$.
Let $k \in \bN$;
if $d$ is even, assume additionally that $k \leq d/2$.
Given $u \in C^\infty(X)$, let $\cu = \cu(x,\rho)$ be such that $\cu(\cdot,0)=u$.
Then $\cu$ is an \defn{extension} of $u$ to $\cmG$.
Graham, Jenne, Mason and Sparling proved~\cite{GJMS1992}*{Proposition~2.1} that the GJMS operator
\begin{equation}
 \label{eqn:defn-gjms}
 P_{2k}^gu := \left. (-\Delta_{\cg})^k\left( t^{-\frac{d-2k}{2}}\cu \right) \right|_{t=1,\rho=0}
\end{equation}
is independent of the choice of extension.
In particular, $P_{2k} \colon C^\infty(X) \to C^\infty(X)$ satisfies the conformal transformation law~\eqref{eqn:conformally-covariant}.
It is known~\cite{GJMS1992}*{Equation~(3.5)} that if $\cu=\cu(x,\rho)$ and $w \in \bR$, then
\begin{equation}
 \label{eqn:general-laplace-formula}
 \Delta_{\cg}(t^w\cu) = t^{w-2}\left( -2\rho \cu^{\prime\prime} + (d + 2w - 2 - \rho g^{ij}g_{ij}^\prime)\cu^\prime + \Delta_{g_\rho} u + \frac{w}{2}g^{ij}g_{ij}^\prime \cu \right) ,
\end{equation}
where primes denote derivatives with respect to $\rho$.
Thus $P_{2k}^g \equiv (-\Delta_g)^k$ modulo lower-order terms.
The dependence of $P_{2k}^g$ on the Taylor expansion of $g_\rho$ in Equation~\eqref{eqn:general-laplace-formula} is responsible for the restriction $k \leq d/2$ when $d$ is even.

\begin{remark}
 \label{lcf-remove-restriction}
 The GJMS operator $P_{2k}$ cannot~\cites{Graham1992,GoverHirachi2004} in general be defined on a manifold of even dimension $d < 2k$.
 However, it can be defined to all orders on conformally Einstein~\cites{Gover2006q,FeffermanGraham2012} and on locally conformally flat~\cites{FeffermanGraham2012,Branson1995} manifolds.
 The point is that, in these situations, one can canonically define~\cite{FeffermanGraham2012}*{Propositions~7.1 and~7.5} the ambient metric modulo $O(\rho^\infty)$.
 In particular, \cref{factorization} is true for all $k \in \bN$ in the case of a special Einstein product of two locally conformally flat manifolds.
\end{remark}

While it is generally difficult to compute the straight and normal ambient metric for a given Riemannian manifold, there are some cases~\cites{FeffermanGraham2012,GoverLeitner2009,AndersonLeistnerNurowski2020,Nurowski2008,LeistnerNurowski2010,LeistnerNurowski2012} where this has been done.
Of relevance to us is the identification~\cite{GoverLeitner2009}*{Theorem~2.1} of the straight and normal ambient metric $(\cmG,\cg)$ of the special Einstein product $(M^\ell \times N^{d-\ell}, g \oplus h)$.
For notational convenience, set $\mu := \lambda/2$.
Then
\begin{equation}
 \label{eqn:special-ambient}
 \cg = 2\rho \, dt^2 + 2t \, dt \, d\rho + t^2 \left( (1 - \mu\rho )^2g + (1 + \mu\rho )^2 h \right) .
\end{equation}
Specializing Equation~\eqref{eqn:general-laplace-formula} to the ambient metric~\eqref{eqn:special-ambient} yields
\begin{multline}
 \label{eqn:special-laplacian}
 \Delta_{\cg}(t^w\cu) = t^{w-2}\Bigl( -2\rho\cu^{\prime\prime} + \Bigl(d+2w-2 + \frac{2\ell\mu\rho}{1 - \mu\rho} - \frac{2(d-\ell)\mu\rho}{1 + \mu\rho} \Bigr)\cu^\prime \\
  + (1-\mu\rho)^{-2}\Delta_g\cu + (1+\mu\rho)^{-2}\Delta_h\cu - \Bigl(\frac{\ell w\mu}{1-\mu\rho} - \frac{(d-\ell)w\mu}{1+\mu\rho}\Bigr)\cu  \Bigr)
\end{multline}
for all $w \in \bR$ and all $\cu = \cu(x,\rho)$.
Equation~\eqref{eqn:special-laplacian} simplifies for certain choices of extension $\cu$ of $u$ (cf.\ \citelist{ \cite{Matsumoto2013}*{Lemma~4.1} \cite{Khaitan2022}*{Section~5} }).

\begin{lemma}
 \label{key-laplacian-lemma}
 Let $(M^\ell \times N^{d-\ell}, g \oplus h)$ be a special Einstein product.
 Given $w,s \in \bR$ and $u \in C^\infty(M \times N)$, define $\cu \in C^\infty(\cmG)$ by
 \begin{equation*}
  \cu(t,x,\rho) := (1+\mu\rho)^s(1-\mu\rho)^{w-s}u(x) .
 \end{equation*}
 Then
 \begin{align*}
  \Delta_{\cg} (t^w\cu) & = t^{w-2}(1+\mu\rho)^{s-2}(1-\mu\rho)^{w-s}\left( \Delta_h + 2s(d-\ell+s-1)\mu\right)u \\
   & \quad + t^{w-2}(1+\mu\rho)^s(1-\mu\rho)^{w-s-2}\left(\Delta_g - 2(w-s)(\ell+w-s-1)\mu \right)u ,
 \end{align*}
 where we identify $u(t,x,\rho) := u(x)$.
\end{lemma}

\begin{proof}
 Direct computation gives
 \begin{align*}
  \cu^\prime & = s\mu(1+\mu\rho)^{s-1}(1-\mu\rho)^{w-s}u - (w-s)\mu(1+\mu\rho)^s(1-\mu\rho)^{w-s-1}u , \\
  \cu^{\prime\prime} & = s(s-1)\mu^2(1+\mu\rho)^{s-2}(1-\mu\rho)^{w-s}u \\
   & \quad - 2s(w-s)\mu^2(1+\mu\rho)^{s-1}(1-\mu\rho)^{w-s-1}u \\
   & \quad + (w-s)(w-s-1)\mu^2(1+\mu\rho)^s(1-\mu\rho)^{w-s-2}u .
 \end{align*}
 Inserting this into Equation~\eqref{eqn:special-laplacian} and simplifying yields the claimed formula.
\end{proof}

\section{The factorization formula}
\label{sec:proof}

Iterating \cref{key-laplacian-lemma} yields our first formula for the GJMS operators of a special Einstein product.

\begin{proposition}
 \label{main-thm}
 Let $(M^\ell \times N^{d-\ell}, g \oplus h)$ be a special Einstein product.
 Let $k \in \bN$ and $s \in \bR$;
 if $d$ is even, assume additionally that $k \leq \frac{d}{2}$.
 Then
 \begin{equation}
  \label{eqn:main-factorization}
  P_{2k}^{g \oplus h} = \sum_{j=0}^k \binom{k}{j} D^{g,\ell,-\lambda}_{k-j,-\frac{d}{2}+k-s} D^{h,d-\ell,\lambda}_{j,s} ,
 \end{equation}
 where
 \begin{equation}
  \label{eqn:D}
  D_{j,s}^{\gamma,m,\mu} := \prod_{t=0}^{j-1} \left( -\Delta_\gamma + (2t-s)(m+s-2t-1)\mu \right) ,
 \end{equation}
 with the convention that the empty product is multiplication by $1$.
\end{proposition}

\begin{proof}
 Let $u \in C^\infty(M \times N)$.
 Set $\mu := \lambda/2$.
 Consider the extension
 \begin{equation*}
  \cu := (1+\mu\rho)^s(1-\mu\rho)^{w-s}u , \qquad w := -\frac{d-2k}{2} .
 \end{equation*}
 Combining \cref{key-laplacian-lemma} with the fact that $\Delta_g$ and $\Delta_h$ commute yields
 \begin{equation*}
  (-\Delta_{\cg})^k(t^w\cu) = t^{w-2k}\sum_{j=0}^k \binom{k}{j}(1+\mu\rho)^{s-2j}(1-\mu\rho)^{w-s-2k+2j}D^{g,\ell,-2\mu}_{k-j,w-s}D^{h,d-\ell,2\mu}_{j,s}u .
 \end{equation*}
 The final conclusion follows from the definition~\eqref{eqn:defn-gjms} of $P_{2k}$.
\end{proof}

The parameter $s$ in \cref{main-thm} is arbitrary;
any other choice will recover the same operator.
Since $D_{j,0}^{\gamma,d,\mu}$, $j \geq 1$, annihilates constants, choosing $s = 0$ and restricting Equation~\eqref{eqn:main-factorization} to functions which depend only on $M$ recovers the factorization~\eqref{eqn:case-chang-factorization}.

The rest of this section requires some functional equations satisfied by the Gamma function $\Gamma$.
The first are the consequences
\begin{align}
 \label{eqn:basic-identities}
 \frac{\Gamma(x+m)}{\Gamma(x)} & = \prod_{t=0}^{m-1} (x+t) , & \frac{\Gamma(x+1)}{\Gamma(x-m+1)} & = \prod_{t=0}^{m-1} (x-t)
\end{align}
of the fundamental identity $z\Gamma(z) = \Gamma(z+1)$.
The second is  Euler's reflection formula~\cite{NIST:DLMF}*{Equation~(5.5.3)}
\begin{equation*}
 \Gamma(z)\Gamma(1-z) = \frac{\pi}{\sin(\pi z)} .
\end{equation*}
In particular, if $m \in \bZ$, then
\begin{equation}
 \label{eqn:euler}
 \frac{\Gamma(1-z)}{\Gamma(1-m-z)} = (-1)^m\frac{\Gamma(z+m)}{\Gamma(z)} .
\end{equation}
The former identities~\eqref{eqn:basic-identities} allow us to express the operator~\eqref{eqn:D} as a ratio of Gamma functions.

\begin{lemma}
 \label{Gamma}
 Let $\gamma$ be a Riemannian metric on an $m$-dimensional manifold, let $s \in \bR$, let $\mu \in \bR \setminus \{0\}$, and let $j \in \bN_0$.
 Then
 \begin{equation*}
  D_{j,s}^{\gamma,m,\mu} = (4\mu)^j\frac{\Gamma\bigl( \frac{1}{2\sqrt{\mu}}B + \frac{4j-m-2s+1}{4} \bigr)}{\Gamma\bigl( \frac{1}{2\sqrt{\mu}}B - \frac{m+2s-1}{4} \bigr)}\frac{\Gamma\bigl( \frac{1}{2\sqrt{\mu}}B + \frac{m+2s+3}{4} \bigr)}{\Gamma\bigl( \frac{1}{2\sqrt{\mu}}B - \frac{4j - m - 2s - 3}{4} \bigr)} ,
 \end{equation*}
 where
 \begin{equation*}
  B := \sqrt{ -\Delta_\gamma + \left( \frac{m-1}{2} \right)^2\mu } ,
 \end{equation*}
 and we adopt the convention $\sqrt{-a} := i\sqrt{a}$ for $a>0$.
\end{lemma}

\begin{proof}
 We first write
 \begin{equation}
  \label{eqn:D-to-B}
  \begin{split}
  D_{j,s}^{\gamma,m,\mu} & = (4\mu)^j\prod_{t=0}^{j-1} \left( -\frac{1}{4\mu}\Delta_\gamma + \left( t - \frac{s}{2} \right)\left( \frac{m+s-1}{2} - t \right) \right) \\
   & = (4\mu)^j\prod_{t=0}^{j-1}\left( \frac{1}{4\mu}B^2 - \left( t - \frac{m+2s-1}{4}\right)^2 \right) .
  \end{split}
 \end{equation}
 The conclusion follows from the identities~\eqref{eqn:basic-identities}.
\end{proof}

One way to remove the free parameter $s$ from \cref{main-thm} is to express the GJMS operators spectrally.
The following lemma recovers the spectral formula of Ao et al.~\cite{AoChanDelatorreFontelosGonzalezWei2019} for the GJMS operators of locally conformally flat special Einstein products.

\begin{corollary}
 \label{mar-formula}
 Let $(M^\ell \times N^{d-\ell}, g \oplus h)$ be a special Einstein product.
 Let $k \in \bN$;
 if $d$ is even, assume additionally that $k \leq \frac{d}{2}$.
 Then
 \begin{equation}
  \label{eqn:mar-formula}
  P_{2k}^{g \oplus h} = (4\lambda)^k \frac{\Gamma\bigl(\frac{1}{2\sqrt{\lambda}}C - \frac{i}{2\sqrt{\lambda}}B + \frac{k+1}{2}\bigr) \Gamma\bigl( \frac{1}{2\sqrt{\lambda}}C + \frac{i}{2\sqrt{\lambda}}B + \frac{k+1}{2}\bigr)}{\Gamma\bigl(\frac{1}{2\sqrt{\lambda}}C - \frac{i}{2\sqrt{\lambda}}B - \frac{k-1}{2}\bigr) \Gamma\bigl( \frac{1}{2\sqrt{\lambda}}C + \frac{i}{2\sqrt{\lambda}}B - \frac{k-1}{2}\bigr)} ,
 \end{equation}
 where
 \begin{align*}
  B & := \sqrt{-\Delta_g - \left( \frac{\ell-1}{2} \right)^2\lambda} , \\
  C & := \sqrt{-\Delta_h + \left( \frac{d-\ell-1}{2} \right)^2\lambda} .
 \end{align*}
\end{corollary}

\begin{proof}
 Since $(N^{d-\ell},h)$ is compact, the Spectral Theorem implies that we need only verify Equation~\eqref{eqn:mar-formula} when applied to functions $uf$, where $f \in C^\infty(M)$ is arbitrary and $u \in C^\infty(N)$ is an eigenfunction of $\Delta_h$.
 Choose $s \in \bR$ such that
 \begin{equation}
  \label{eqn:choose-ell}
  Cu = \frac{d-\ell+2s-1}{2}\sqrt{\lambda}u ;
 \end{equation}
 that $s$ exists follows from the positivity of the operator $C$.
 Equation~\eqref{eqn:D-to-B} implies that $D_{j,s}^{h,d-\ell,\lambda}u = 0$ for all $j \geq 1$.
 \Cref{main-thm} then implies that
 \begin{equation*}
  P_{2k}^{g \oplus h}(uf) = D_{k,-\frac{d}{2}+k-s}^{g,\ell,-\lambda}(uf) .
 \end{equation*}
 Combining this with \cref{Gamma} yields
 \begin{align*}
  P_{2k}^{g \oplus h}(uf) & = (-4\lambda)^k\frac{\Gamma\bigl(-\frac{i}{2\sqrt{\lambda}}B + \frac{d-\ell+2k+2s+1}{4} \bigr) \Gamma\bigl( -\frac{i}{2\sqrt{\lambda}}B - \frac{d-\ell-2k+2s-3}{4} \bigr)}{\Gamma\bigl( -\frac{i}{2\sqrt{\lambda}}B + \frac{d-\ell-2k+2s+1}{4} \bigr) \Gamma\bigl( -\frac{i}{2\sqrt{\lambda}}B - \frac{d-\ell+2k+2s-3}{4} \bigr)}(uf) \\
  & = (-4\lambda)^k\frac{\Gamma\bigl(-\frac{i}{2\sqrt{\lambda}}B + \frac{1}{2\sqrt{\lambda}}C + \frac{k+1}{2} \bigr) \Gamma\bigl( -\frac{i}{2\sqrt{\lambda}}B - \frac{1}{2\sqrt{\lambda}}C + \frac{k+1}{2} \bigr)}{\Gamma\bigl( -\frac{i}{2\sqrt{\lambda}}B + \frac{1}{2\sqrt{\lambda}}C - \frac{k-1}{2} \bigr) \Gamma\bigl( -\frac{i}{2\sqrt{\lambda}}B - \frac{1}{2\sqrt{\lambda}}C - \frac{k-1}{2} \bigr)}(uf) ,
 \end{align*}
 where the second equality uses Equation~\eqref{eqn:choose-ell}.
 The final conclusion follows from Euler's reflection formula~\eqref{eqn:euler}.
\end{proof}

Another way to remove the free parameter $s$ from \cref{main-thm} is to express the GJMS operators as compositions of second- and fourth-order operators which are real polynomials in $\Delta_g$ and $\Delta_h$.
This is done by combining the basic identities~\eqref{eqn:basic-identities} with Equation~\eqref{eqn:mar-formula} and eliminating square roots.

\begin{proof}[Proof of \cref{factorization}]
 Combining \cref{mar-formula} with the identities~\eqref{eqn:basic-identities} yields
 \begin{align*}
  P_{2k}^{g \oplus h} & = (4\lambda)^k\prod_{t=0}^{k-1} \left( \frac{1}{2\sqrt{\lambda}}C - \frac{i}{2\sqrt{\lambda}}B + t - \frac{k-1}{2} \right) \left( \frac{1}{2\sqrt{\lambda}}C + \frac{i}{2\sqrt{\lambda}}B + t - \frac{k-1}{2}\right) \\
  & = \prod_{t=0}^{k-1} \left( B^2 + C^2 - 2(k-1-2t)\sqrt{\lambda}C + (k-1-2t)^2\lambda \right) .
 \end{align*}
 Denote
 \begin{equation*}
  \mD_t := B^2 + C^2 - 2(k-1-2t)\sqrt{\lambda}C + (k-1-2t)^2\lambda ,
 \end{equation*}
 so that
 \begin{equation*}
  P_{2k}^{g \oplus h} = \prod_{t=0}^{k-1} \mD_t .
 \end{equation*}
 The conclusion follows from the observations
 \begin{align*}
  \mD_t\mD_{k-1-t} & = \left( B^2 + C^2 + (k-1-2t)^2\lambda \right)^2 - 4(k-1-2t)^2\lambda C^2 , \\
  \mD_{\frac{k-1}{2}} & = B^2 + C^2 , && \text{if $k$ is odd}. \qedhere
 \end{align*}
\end{proof}

\section{Applications}
\label{sec:maximum}

We now turn to the problem of constructing, under suitable assumptions on the parameters $d$, $\ell$, and $k$, conformal metrics which realize the $Q_{2k}$-Yamabe constant~\eqref{eqn:yamabe} of a special Einstein product.

The first task is to find sufficient conditions for $P_{2k}$ to satisfy the Strong Maximum Principle.
To that end, we first use an iterative strategy modeled on an argument of Gursky and Malchiodi~\cite{GurskyMalchiodi2012} to derive sufficient conditions for certain fourth-order operators to satisfy the Strong Maximum Principle.
Unlike their argument, our argument does not pass through geometric scalar quantities.

\begin{lemma}
 \label{baby-maximum-principle}
 Let $(X^d,g)$, $d \geq 5$, be a closed Riemannian manifold with constant positive scalar curvature.
 Given constants $a,c \in \bR$ such that $a < \frac{d(d-2)}{4}$, define $D \colon C^\infty(X) \to C^\infty(X)$ by
 \begin{equation*}
  Du := \Delta^2u + 4a\lp P,\nabla^2u\rp - (d-2)J\Delta u + cu
 \end{equation*}
 for all $u \in C^\infty(X)$, where $P := \frac{1}{d-2}(\Ric - Jg)$ is the Schouten tensor and $J := \tr_g P$ is its trace.
 Suppose that $c>0$ and
 \begin{equation}
  \label{eqn:assumption}
  \begin{split}
   (1-a)\Ric + \frac{2(d^2-5d+2)a-(d-2)^2}{d^2-4d-4}Jg & \geq 0 , \\
   \frac{(d-4)(d^2-2d-4a)^2}{4d(d^2-4d-4)}J^2 - c - (d-4)a^2\lv\mathring{P}\rv^2 \geq 0 ,
  \end{split}
 \end{equation}
 where $\mathring{P} := P - \frac{J}{d}g$.
 Then $D$ satisfies the Strong Maximum Principle.
\end{lemma}

\begin{proof}
 Let $e < 0$ be a constant to be specified later.
 Let $u \in C^\infty(X)$ be such that $Du \geq 0$ and $\min u \leq 0$.
 Set
 \begin{equation*}
  \lambda_0 := \frac{1}{1-\min u} \in (0,1] .
 \end{equation*}
 Given $\lambda \in (0,\lambda_0)$, set $u_\lambda := 1-\lambda + \lambda u$.
 Then $\min u_\lambda > 0$ and
 \begin{equation*}
  Du_\lambda = (1-\lambda)c + \lambda Du > 0 .
 \end{equation*}
 Define $S_\lambda \in C^\infty(X)$ by
 \begin{equation*}
  S_\lambda := \Delta u_\lambda + \frac{2}{d-4}u_\lambda^{-1}\lv\nabla u_\lambda\rv^2 + eu_\lambda .
 \end{equation*}
 It suffices to show that $S_\lambda < 0$ for all $\lambda \in (0,\lambda_0)$.
 Indeed, if this is the case, then
 \begin{equation*}
  0 > \frac{d-2}{d-4}S_\lambda u_\lambda^{\frac{2}{d-4}} = \Delta u_\lambda^{\frac{d-2}{d-4}} + \frac{d-2}{d-4}eu_\lambda^{\frac{d-2}{d-4}} .
 \end{equation*}
 Taking the limit $\lambda \to \lambda_0$ yields
 \begin{equation*}
  \Delta u_{\lambda_0}^{\frac{d-2}{d-4}} + \frac{d-2}{d-4}eu_{\lambda_0}^{\frac{d-2}{d-4}} \leq 0 .
 \end{equation*}
 Since $e<0$ and $\min u_{\lambda_0}=0$, the Strong Maximum Principle for second-order elliptic operators implies that $u_{\lambda_0} = 0$.
 We conclude that
 \begin{equation*}
  0 = Du_{\lambda_0} \geq (1-\lambda_0)c,
 \end{equation*}
 and hence $\lambda_0=1$.
 In particular, $u=u_{\lambda_0}=0$.
 
 We now prove that $S_\lambda < 0$ for all $\lambda \in (0,\lambda_0)$.
 Suppose to the contrary that the claim is false.
 Note that $S_0 = e < 0$.
 A standard continuity argument implies that there is a $\lambda \in (0,\lambda_0)$ such that $\max S_\lambda = 0$.
 Pick $x \in X$ such that $S_\lambda(x)=0$.
 Denote $u := u_\lambda$ and $S := S_\lambda$.
 Define $\Box_u \colon C^\infty(X) \to C^\infty(X)$ by
 \begin{equation*}
  \Box_u w := \Delta w - \frac{4}{d-4}u^{-1}\lp \nabla u, \nabla w \rp .
 \end{equation*}
 Combining our assumptions on $S$ and $x$ with the Bochner formula yields
 \begin{align*}
  0 & \geq (\Box_u S)(x) \\
   & = \Delta^2u + \frac{4}{d-4}u^{-1}\lv\nabla^2u\rv^2 - \frac{8(d-2)}{(d-4)^2}u^{-2}\nabla^2u(\nabla u,\nabla u) + \frac{4}{d-4}u^{-1}\Ric(\nabla u,\nabla u) \\
    & \quad + \frac{4(d-1)}{(d-4)^2}u^{-3}\lv\nabla u\rv^4 - \frac{4}{d-4}eu^{-1}\lv\nabla u\rv^2 - e^2u ,
 \end{align*}
 where the right-hand side is evaluated at $x$.
 Since $Du > 0$, we deduce that
 \begin{align*}
  0 & > -4a\lp P,\nabla^2u\rp+ \frac{4}{d-4}u^{-1}\lv\nabla^2u\rv^2 - \frac{8(d-2)}{(d-4)^2}u^{-2}\nabla^2u(\nabla u,\nabla u) \\
   & \quad + \frac{4}{d-4}u^{-1}\Ric(\nabla u,\nabla u) + \frac{4(d-1)}{(d-4)^2}u^{-3}\lv\nabla u\rv^4 - (e^2+c)u \\
   & \quad - \frac{4}{d-4}\left( e + \frac{d-2}{2}J \right)u^{-1}\lv\nabla u\rv^2 - (d-2)Jeu .
 \end{align*}
 Regrouping terms yields
 \begin{align*}
  0 & > \frac{4u^{-1}}{d-4}\left| -\frac{d-4}{2}a uP + \nabla^2u - \frac{d-2}{d-4}u^{-1}(du)^2 + \frac{1}{d-4}u^{-1}\lv\nabla u\rv^2 g \right|^2 \\
   & \quad + \frac{4(1-a)}{d-4}u^{-1}\Ric(\nabla u,\nabla u) - \frac{4}{d-4}\left( \frac{d-6}{d-4}e + \frac{d-2-4a}{2}J \right)u^{-1}\lv\nabla u\rv^2 \\
   & \quad - \left( e^2 + (d-2)Je + c + (d-4)a^2\lv P\rv^2 \right)u .
 \end{align*}
 Recalling that $S=0$, we deduce from the Cauchy--Schwarz inequality that
 \begin{multline*}
  \left| -\frac{d-4}{2}a uP + \nabla^2u - \frac{d-2}{d-4}u^{-1}(du)^2 + \frac{1}{d-4}u^{-1}\lv\nabla u\rv^2 g \right|^2 \\
   \geq \frac{1}{d}\left( -\frac{d-4}{2}aJ - e \right)^2u^2 .
 \end{multline*}
 Combining the previous two displays gives
 \begin{align*}
  0 & > \frac{4(1-a)}{d-4}u^{-1}\Ric(\nabla u,\nabla u) - \frac{4}{d-4}\left( \frac{d-6}{d-4}e + \frac{d-2-4a}{2}J \right)u^{-1}\lv\nabla u\rv^2 \\
   & \quad - \left( \frac{d^2-4d-4}{d(d-4)}e^2 + \frac{d^2-2d-4a}{d}Je + c + (d-4)a^2\lv\mathring{P}\rv^2 \right)u .
 \end{align*}
 Since $d \geq 5$, we see that $d^2-4d-4 \geq 1$.
 Taking $e = -\frac{(d^2-2d-4a)(d-4)}{2(d^2-4d-4)}J$ in the previous display yields
 \begin{align*}
  0 & > \frac{4}{d-4}u^{-1} \left\lp (1-a)\Ric + \frac{2(d^2-5d+2)a-(d-2)^2}{d^2-4d-4}Jg , (du)^2 \right\rp \\
   & \quad + \left( \frac{(d-4)(d^2-2d-4a)^2}{4d(d^2-4d-4)}J^2 - c - (d-4)a^2\lv\mathring{P}\rv^2 \right)u .
 \end{align*}
 This contradicts our assumption~\eqref{eqn:assumption}.
\end{proof}

\Cref{baby-maximum-principle} leads to sufficient conditions for the GJMS operator $P_{2k}$ to satisfy the Strong Maximum Principle.

\begin{proof}[Proof of \cref{positivity}]
 If $\ell=0$, then the conclusion follows from the factorization~\eqref{eqn:einstein-factorization} of the GJMS operators at Einstein manifolds.
 Suppose now that $\ell\geq1$.

 By scaling, we may assume that $\lambda=1$.
 In this case the Schouten tensor of $\cg := g \oplus h$ is $P = \frac{1}{2}\bigl( (-g) \oplus h \bigr)$.
 Direct computation gives
 \begin{equation*}
  J = \frac{d-2\ell}{2} , \quad \Ric \leq \frac{2(d-\ell-1)}{d-2\ell}J\cg , \quad \lv\mathring{P}\rv^2 = \frac{\ell(d-\ell)}{d} .
 \end{equation*}
 It follows that the fourth-order operator~\eqref{eqn:defn-Dt} is equivalently written
 \begin{align*}
  D_tu & = \Delta^2u + 4a\lp P,\nabla^2u\rp - (d-2)J\Delta u + cu , \\
  c & = \left( \left( \frac{d-2}{2} \right)^2 - a \right) \left( \left( \frac{d-2\ell}{2} \right)^2 - a \right) ,
 \end{align*}
 with $a = t^2 \in \bigl[ 1, (k-1)^2 \bigr]$.
 Taking $d \geq 2k + 2\ell - 1$ yields $a < \frac{d(d-2)}{4}$.
 Since $a = O(1)$ as $d \to \infty$, we see that
 \begin{equation}
  \label{eqn:first-equation}
  \begin{split}
  \MoveEqLeft[4] (1-a)\Ric + \frac{2(d^2-5d+2)a-(d-2)^2}{d^2-4d-4}J\cg \\
   & \geq \biggl( \frac{2(d-\ell-1)(1-a)}{d-2\ell} + \frac{2(d^2-5d+2)a - (d-2)^2}{d^2-4d-4} \biggr)J\cg 
  \end{split}
 \end{equation}
 and
 \begin{equation}
  \label{eqn:second-equation}
  \begin{split}
   \MoveEqLeft[4] \frac{(d-4)(d^2-2d-4a)^2}{4d(d^2-4d-4)}J^2 - c - (d-4)a^2\lv\mathring{P}\rv^2 \\
    & = \frac{(d-2)^2}{d^2-4d-4}J^2 + \left( \frac{(d-2)^2}{(d-2\ell)^2} + 1 - \frac{2(d-2)(d-4)}{d^2-4d-4} \right)aJ^2 \\
     & \quad + \left( \frac{(d-4)(d-2\ell)^2}{d(d^2-4d-4)} - 1 - \frac{\ell(d-\ell)(d-4)}{d}\right) a^2 .
  \end{split}
 \end{equation}
 In particular,
 \begin{align*}
  (1-a)\Ric + \frac{2(d^2-5d+2)a-(d-2)^2}{d^2-4d-4}J\cg & \geq \frac{d}{2}\cg + O(1) , \\
  \frac{(d-4)(d^2-2d-4a)^2}{4d(d^2-4d-4)}J^2 - c - (d-4)a^2\lv\mathring{P}\rv^2 & = J^2 + O(d) .
 \end{align*}
 The conclusion readily follows.
\end{proof}

\begin{remark}
 \label{P6-remark}
 In principle, the number $D(k,\ell)$ can be estimated using the roots of the coefficients of $J\cg$ and $J^2$ in Equations~\eqref{eqn:first-equation} and~\eqref{eqn:second-equation}, respectively, with $a \in \bigl[1,(k-1)^2\bigr]$.
 For example, taking $k=3$, $a=(k-1)^2=4$, and $\ell=1$ yields
 \begin{align*}
  (1-a)\Ric + \frac{2(d^2-5d+2)a-(d-2)^2}{d^2-4d-4}J\cg & \geq \frac{(d-6)^2}{d^2-4d-4}J\cg , \\
  \frac{(d-4)(d^2-2d-4a)^2}{4d(d^2-4d-4)}J^2 - c - (d-4)a^2\lv\mathring{P}\rv^2 & \geq \frac{d^2-52d+228}{d^2-4d-4}J^2 .
 \end{align*}
 The right-hand sides are positive for $d \geq 48$, and hence $7 \leq D(3,1) \leq 48$.
\end{remark}

The second task is to verify the estimate $Y_{2k} < Y_{2k}(S^d,[\grd])$ on special Einstein products $M^\ell \times N^{d-\ell}$ with $d$ sufficiently large.
In the case of locally conformally flat special Einstein products, we do this using results of Qing and Raske~\cite{QingRaske2006}.
Indeed, in this case we know the optimal range of dimensions $d$.

\begin{lemma}
 \label{positive-mass}
 Let $k,\ell$ be nonnegative integers and let $d \geq 2k + 2\ell - 1$.
 Suppose that $(M^\ell \times N^{d-\ell}, g \oplus h)$ is a locally conformally flat special Einstein product.
 Then
 \begin{equation}
  \label{eqn:yamabe-estimate}
  Y_{2k}(M^\ell \times N^{d-\ell} , [g \oplus h] ) < Y_{2k}(S^d, \grd) .
 \end{equation}
 Moreover, there is a $\sigma \in [g \oplus h]$ such that $Q_{2k}^\sigma = Y_{2k}(M^\ell \times N^{d-\ell},[g \oplus h])$ and $\Vol_\sigma(M \times N) = 1$.
\end{lemma}

\begin{proof}
 Denote $X^d := M^\ell \times N^{d-\ell}$ and $\sigma_0 := g \oplus h$.
 Since $(X^d,\sigma_0)$ is locally conformally flat, both $(M^\ell,g)$ and $(N^{d-\ell},h)$ are locally conformally flat~\cite{Besse}*{Example~1.167(3)}.
 Therefore their Riemannian universal covers are, up to isometry, the $\ell$-dimensional hyperbolic space $(H^\ell,\ghyp)$ and the $(d-\ell)$-dimensional sphere $(S^{d-\ell},\grd)$, respectively.
 
 Let $\pi \colon (\cX^d,\csigma_0) \to (X^d,\sigma_0)$ denote the Riemannian universal cover of $X$.
 Recall the model~\cite{MazzeoSmale1991}*{p.\ 583}
 \begin{equation*}
  (H^\ell \times S^{d-\ell}, \ghyp \oplus \grd) = \bigl( \bR^d \setminus \bR^{\ell-1} , \lv y\rv^{-2}(dx^2 \oplus dy^2) \bigr) ,
 \end{equation*}
 where $(x,y) \in \bR^{\ell-1} \times \bR^{d-\ell+1}$ and we write $dx^2$ and $dy^2$ for the Euclidean metrics on the respective factors.
 By the uniqueness of universal covers, we deduce that, up to isometry,
 \begin{equation*}
  \bigl( \cX , \csigma_0 \bigr) = \bigl( \bR^d \setminus \bR^{\ell-1} , \lv y\rv^{-2}(dx^2 \oplus dy^2) \bigr) .
 \end{equation*}
 Since stereographic projection from a point in $S^{\ell-1}$ maps $S^d \setminus S^{\ell-1}$ onto $\bR^d \setminus \bR^{\ell-1}$, we deduce that $\pi_1(X)$ is a Kleinian group with singular set equal to $S^{\ell-1}$.
 Since $\pi_1(X)$ is cocompact, its Poincar\'e exponent $\delta(\pi_1(X))$ equals the dimension of its singular set~\cite{Nicholls1989}*{Theorem~9.3.6};
 i.e.\ $\delta(\pi_1(X)) = \ell-1 < \frac{d-2k}{2}$.
 
 Now recall~\cite{GazzolaGrunauSweers2010}*{p.\ 48} that the fundamental solution for the GJMS operator $P_{2k}^{\gfl}$ of order $2k$ on flat Euclidean space $(\bR^d,\gfl)$, $d > 2k$, is
 \begin{equation}
  \label{eqn:standard-fundamental-solution}
  G_{2k}^{\gfl}(\zeta,\xi) = c_{n,k}\lv \zeta - \xi \rv^{2k-d} , \qquad c_{n,k} = \frac{\Gamma\bigl(\frac{d}{2}-k\bigr)}{4^{k}(k-1)!\pi^{d/2}} .
 \end{equation}
 Conformal covariance~\eqref{eqn:conformally-covariant} then implies that
 \begin{equation*}
  G_{2k}^{\csigma_0}(\zeta,\xi) = c_{n,k} \lv y(\zeta) \rv^{-\frac{d-2k}{2}} \lv \zeta - \xi \rv^{2k-d} \lv y(\xi) \rv^{-\frac{d-2k}{2}}
 \end{equation*}
 is such that $G_{2k,\zeta}^{\csigma_0} := G_{2k}^{\csigma_0}(\zeta,\cdot)$ satisfies $P_{2k}^{\csigma_0}G_{2k,\zeta}^{\csigma_0} = \delta_{\zeta}$ in the distributional sense with respect to the Riemannian volume element of $\csigma_0$.
 Since the Poincar\'e exponent of $\pi_1(X)$ satisfies $\delta(\pi_1(X)) < \frac{d-2k}{2}$, we deduce~\cite{QingRaske2006}*{Theorem~2.1} that
 \begin{equation}
  \label{eqn:defn-K}
  G_{2k}^{\sigma_0}(\zeta,\xi) := \sum_{\tau \in \pi_1(X)} G_{2k}^{\cgamma}(\zeta,\tau\xi)
 \end{equation}
 converges and is equal to the Green's function for $P_{2k}^{\sigma_0}$.
 Clearly $G_{2k}^{\sigma_0} > 0$.
 
 Fix $\zeta \in X^d$.
 Pick an evenly-covered neighborhood $U \subset X^d$ of $\zeta$ and a diffeomorphism $\varphi \colon U \to \varphi(U) \subset \bR^d$.
 Then the metric $\varphi^\ast (dx^2 \oplus dy^2)$ and the standard coordinates on $U$ give conformal normal coordinates around $\zeta$.
 We deduce from Equations~\eqref{eqn:standard-fundamental-solution} and~\eqref{eqn:defn-K} that
 \begin{equation*}
  G_{2k}^{\varphi^\ast (dx^2 \oplus dy^2)}(\zeta,\xi) = c_{n,k}\lv \zeta - \xi \rv^{2k-d} + \mu(\zeta) + O(1), \qquad \mu(\zeta) > 0 ,
 \end{equation*}
 for $\lv \xi - \zeta \rv$ small.
 Therefore~\cite{MazumdarVetois2020}*{Proposition~3.1} Estimate~\eqref{eqn:yamabe-estimate} holds.
 
 Finally, \cref{factorization} implies that $P_{2k}^{\sigma_0}>0$ (cf.\ \cite{QingRaske2006}*{Theorem~3.2}).
 Since $G_{2k}^{\sigma_0}>0$, we conclude using a result~\cite{Mazumdar2016}*{Theorem~3} of Mazumdar that there is a positive function $u \in C^\infty(M \times N)$ such that
 \begin{align*}
  \int_{M \times N} u \, P_{2k}^{\sigma_0} u \dvol_{\sigma_0} & = \frac{d-2k}{2}Y_{2k}(M^\ell \otimes N^{d-\ell}, [\sigma_0] ) , \\
  \int_{M \times N} u^{\frac{d}{d-2k}} \dvol_{\sigma_0} & = 1 .
 \end{align*}
 Then $\sigma := u^{\frac{4}{d-2k}}\sigma_0$ is the desired metric.
\end{proof}

Combining \cref{positivity,positive-mass} with the aformentioned result of Mazumdar~\cite{Mazumdar2016} yields our main existence theorem.

\begin{proof}[Proof of \cref{solvability}]
 Let $D := D(k,\ell)$ be the maximum of $2k+4$ and the constant from \cref{positivity}.
 Since $D \geq 2k+2\ell-1$, \cref{positive-mass} implies that we need only consider the case when $g \oplus h$ is not locally conformally flat.
 Since $D \geq 2k+4$, a higher-order version of Aubin's local test functions~\cite{MazumdarVetois2020}*{Proposition~2.1} implies that Estimate~\eqref{eqn:yamabe-estimate} holds.
 \Cref{factorization} implies that $P_{2k}^{g \oplus h}>0$, and \cref{positivity} implies that $G_{2k}^{g \oplus h}>0$.
 The conclusion now follows as in the last paragraph of the proof of \cref{positive-mass}.
\end{proof}

\section*{Acknowledgments}
JSC thanks Mar\'ia del Mar Gonz\'alez for directing us to her work~\cites{BettiolGonzalezMaalaoui2023,AoChanDelatorreFontelosGonzalezWei2019} on the fractional GJMS operators of the special Einstein product $H^\ell \times S^{d-\ell}$.

JSC was partially supported by the Simons Foundation (Grant \#524601), and by the Simons Foundation and the Mathematisches Forschungsinstitut Oberwolfach via a Simons Visiting Professorship.
AM is supported by the project \emph{Geometric Problems with Loss of Compactness} from Scuola Normale Superiore. He is also a member of GNAMPA as part of INdAM.

\bibliographystyle{abbrv}
\bibliography{bib}
\end{document}